\documentclass[11pt]{article}

\usepackage{amsmath,amssymb,amsthm,microtype,geometry}

\newtheorem{theorem}{Theorem}[section]

\newtheorem{lemma}{Lemma}[section]

\newtheorem{proposition}{Proposition}[section]
\newtheorem{conjecture}{Conjecture}[section]

\title{Norm rigidity and equality cases for the Dyn--Farkhi inequality}
\author{Mark Meyer}

\begin{document}

\maketitle

\begin{abstract}
    For a convex body $K\subset\mathbb{R}^2$ that is symmetric with respect to the origin, and for a nonempty set $S\subset\mathbb{R}^2$, we study the $K$-Hausdorff distance from convex hull, defined by
    \begin{align*}
        d^{(K)}(S):=\sup_{x\in \textup{conv}(S)}\inf_{s\in S}\|x-s\|_K,
    \end{align*}
    where $\|\cdot \|_K$ is the norm whose closed unit ball is $K$. We consider the problem of characterizing the origin symmetric convex bodies $K$ for which 
    \begin{align*}
        d^{(K)}(A+B)^2\leq d^{(K)}(A)^2+d^{(K)}(B)^2
    \end{align*}
    holds for all nonempty compact $A,B\subset\mathbb{R}^2$. We solve this problem, proving that this property holds if and only if $K$ is an ellipse centered at $0$. We then characterize the conditions for equality for this bound when $K$ is an ellipse. 
\end{abstract}


\section{Introduction}

In the paper \cite{meyer2}, we proved that if $A,B\subset\mathbb{R}^2$ are nonempty compact sets, then
\begin{align}\label{eq:dyn_farkhi_conjecture}
    d(A+B)^2\leq d(A)^2+d(B)^2,
\end{align}
where $d(S):=d_H(S,\textup{conv}(S))$ is the Hausdorff distance between a nonempty set $S$ and its convex hull, and 
\begin{align*}
    A+B:=\{a+b:a\in A,b\in B\}
\end{align*}
is Minkowski summation. This resolved a conjecture of Dyn and Farkhi \cite{dyn1} in the planar case. It turns out that the bound \eqref{eq:dyn_farkhi_conjecture} is false when $n\geq 3$, as proved by Fradelizi, Madiman, Marsiglietti, and Zvavitch in their survey \cite{fradelizi1}, and false even in the symmetric case $A=B$, as shown recently by van Hintum in \cite{vanHintum}.

We are interested in studying the rigidity of norms in the context of the Hausdorff distance. In particular, we would like to know which norms preserve the square subadditive bound \eqref{eq:dyn_farkhi_conjecture} for the associated Hausdorff distance. Specifically, a set $K\subset\mathbb{R}^2$ is a convex body if it is compact, convex, with nonempty interior. A convex body $K$ is symmetric about the origin if $K=-K$. The norm whose closed unit ball is the origin symmetric convex body $K$ is defined by
\begin{align*}
    \|x\|_K:=\inf\{\lambda>0:x\in \lambda K\},\qquad x\in \mathbb{R}^2.
\end{align*}
Formally, as done in the survey \cite{fradelizi1}, the $K$-Hausdorff distance from convex hull is defined by
\begin{align*}
    d^{(K)}(S):=\sup_{x\in \textup{conv}(S)}d^{(K)}(x,S),\qquad d^{(K)}(x,S):=\inf_{s\in S}\|x-s\|_K,
\end{align*}
for any nonempty set $S\subset\mathbb{R}^2$. In the case where $K$ is the unit $\ell_2$-ball $B_2^2$, we have the original definition of Hausdorff distance $d^{(B_2^2)}(S)=d(S)$. We will prove the following characterization theorem that shows that among all planar norms, universal square subadditivity of the Hausdorff distance characterizes inner product norms.

\begin{theorem}\label{theorem:norm_rigidity_theorem}
    Let $K\subset \mathbb{R}^2$ be a convex body that is symmetric about the origin. Then
    \begin{align*}
        d^{(K)}(A+B)^2\leq d^{(K)}(A)^2+d^{(K)}(B)^2
    \end{align*}
    holds for all nonempty compact $A$ and $B$ in $\mathbb{R}^2$ if and only if $K$ is an ellipse.
\end{theorem}

The key element of our proof is John's ellipsoid theorem. We show that it is possible to construct an example contradicting the quadratic bound \eqref{eq:dyn_farkhi_conjecture} if $K$ is different from its John ellipsoid.

We will also characterize the equality conditions for the square subadditive bound. In the next theorem, for an ellipse $E$ centered at $0$, two lines $L_1$ and $L_2$ are called $E$-orthogonal if the lines $TL_1$ and $TL_2$ are orthogonal, where $T$ is any invertible linear transformation such that $TE=B_2^2$. Note that this definition is independent of the choice of $T$: if $T_1E=T_2E=B_2^2$, then $T_2T_1^{-1}$ preserves the Euclidean ball, and is therefore orthogonal. That is, $T_1$ and $T_2$ differ by an orthogonal transformation.

\begin{theorem}\label{theorem:equality_conditions}
    Let $E\subset\mathbb{R}^2$ be an ellipse centered at $0$. If $A,B\subset\mathbb{R}^2$ are compact and nonconvex, then
    \begin{align*}
        d^{(E)}(A+B)^2=d^{(E)}(A)^2+d^{(E)}(B)^2
    \end{align*}
    if and only if the affine hulls $\textup{aff}(A)$ and $\textup{aff}(B)$ are $E$-orthogonal lines.
\end{theorem}

As a follow-up of our proof of Theorem \ref{theorem:equality_conditions}, we also make the following observation about the Minkowski sum of a convex set and a nonconvex set. Note that by the dimension of a set $S\subset\mathbb{R}^n$, we mean the dimension of its affine hull $\textup{aff}(S)$. We use $\textup{dim}(S)$ to denote the affine dimension of $S$.

\begin{proposition}\label{proposition:convex+nonconvex}
    Let $E\subset\mathbb{R}^2$ be an ellipse centered at $0$. If $A\subset\mathbb{R}^2$ is a convex set containing more than one point, and if $B\subset\mathbb{R}^2$ is a nonconvex set such that
    \begin{align*}
        d^{(E)}(A+B)=d^{(E)}(B),
    \end{align*}
    then either $\textup{aff}(A)$ and $\textup{aff}(B)$ are $E$-orthogonal lines, or $\textup{dim}(B)=2$.
\end{proposition}

Note that from Theorem \ref{theorem:equality_conditions}, if $A\subset\mathbb{R}^2$ is a compact nonconvex set, then
\begin{align*}
    d\left(\frac{A+A}{2}\right)<\frac{1}{\sqrt{2}}d(A),
\end{align*}
which is the Dyn--Farkhi conjecture with $A=B$, but the bound is strict. The fact that the above bound is strict is of special interest. We conjecture the following improvement.

\begin{conjecture}\label{conjecture:improved_constant_symmetric_case}
    Let $A\subset\mathbb{R}^2$ be a nonconvex compact set. Then
    \begin{align*}
        d\left(\frac{A+A}{2}\right)\leq \frac{1}{2}d(A).
    \end{align*}
\end{conjecture}
We note that in the paper \cite{vanHintum}, van Hintum posed a version of this question for arbitrary dimension $n\geq 2$.

Aside from the above-mentioned work, the Hausdorff distance has been studied in the setting of the $\ell_p$-norm in \cite{M-2026}.

\subsection{Acknowledgements} We thank Matthieu Fradelizi, Robert Fraser, and Buma Fridman for helpful discussions. We thank the laboratory LAMA at University Gustave Eiffel for their hospitality. This work was completed under the financial support of the National Science Foundation through the MSPRF program (award number: 2502794). The author used chat GPT for assistance with proofreading and exposition of the manuscript.

\section{Proof of Theorem \ref{theorem:norm_rigidity_theorem}}

The following lemma is a direct consequence of John's ellipsoid theorem. 

\begin{lemma}\label{lemma:boundary_four_points}
    Let $K\subset\mathbb{R}^2$ be a convex body that is symmetric about the origin and let $\mathcal{E}(K)$ be its John ellipse. Then the intersection
    \begin{align*}
        \partial K\cap \partial \mathcal{E}(K)
    \end{align*}
    contains at least four points.
\end{lemma}

For a convex body $K\subset\mathbb{R}^2$ and $c\in \textup{int}(K)$, define the function $\kappa:\mathbb{R}\rightarrow\partial K$ by
\begin{align*}
    \kappa(\theta):=\frac{\hspace{-0.65cm}(\cos(\theta),\sin(\theta))}{\|(\cos(\theta),\sin(\theta))\|_{K-c}}+c.
\end{align*}
If $x_1,x_2\in \partial K$ are distinct, choose $\theta_1$ such that $\kappa(\theta_1)=x_1$, and choose the smallest $\theta_2>\theta_1$ such that $\kappa(\theta_2)=x_2$. Then the open arc connecting $x_1$ and $x_2$ in the counterclockwise direction is the set
\begin{align*}
    \textup{arc}_K(x_1,x_2):=\{\kappa(\theta):\theta_1< \theta< \theta_2\}.
\end{align*}

The next lemma follows directly from the definition of $\textup{arc}_K(x_1,x_2)$.

\begin{lemma}\label{lemma:arc_lemma}
    Let $K\subset\mathbb{R}^2$ be a convex body and let $E$ be a compact proper subset of $\partial K$ that contains at least two points. Then there exist distinct $x_1,x_2\in E$ such that 
    \begin{align*}
        \textup{arc}_K(x_1,x_2)\cap E=\varnothing.
    \end{align*}
    Moreover, if $K$ and $E$ are symmetric about the origin and if $E$ contains at least three points, then the points can be chosen so that
    \begin{align*}
        x_2\neq -x_1,\qquad x_1\in \textup{arc}_K(-x_2,x_2).
    \end{align*}
\end{lemma}

We also need the following key observation about the intersection of a line through the origin with the boundary of an origin symmetric convex body.

\begin{lemma}\label{lemma:formula_intersection_with_boundary}
    Let $K\subset\mathbb{R}^2$ be a convex body that is symmetric about the origin, and let $p,x\in \partial K$ be such that $-x,p,x$ are distinct and occur in counterclockwise order. Let $L$ be the line that passes through the origin and is parallel to the intervals $[p,-x]$ and $[x,-p]$. Then $L$ intersects $\textup{arc}_K(p,x)$ at exactly one point $k$, and
    \begin{align*}
        p+x=\|p+x\|_Kk.
    \end{align*}
\end{lemma}

\begin{proof}
    The arc $\textup{arc}_K(p,x)$ must lie in the strip that is bounded by the parallel lines through the intervals $[p,-x]$ and $[x,-p]$. Since $p$ and $x$ lie on opposite sides of the midline $L$ of the strip, and since $\kappa$ is continuous, there exists a point $k\in L\cap \textup{arc}_K(p,x)$. The line $L$ intersects $\partial K$ at exactly two points $k$ and $\tilde{k}$, and $\tilde{k}$ lies in the complementary arc $\textup{arc}_K(x,p)$, so the point $k$ is unique.

    Since $[p,-x]$ is parallel to $L$, we have
    \begin{align*}
        p+x\in L=\textup{span}(k).
    \end{align*}
    The fact that $-x,p,x$ occur on $\partial K$ in counterclockwise order implies that $p+x=\beta k$ for some $\beta>0$. Because $k\in \partial K$, the intersection $L\cap \lambda K$ is exactly $[-\lambda k,\lambda k]$. It follows that the smallest $\lambda>0$ such that $p+x\in \lambda K$ is $\beta$, and therefore $\beta=\|p+x\|_K$.
\end{proof}

\begin{proof}[Proof of Theorem \ref{theorem:norm_rigidity_theorem}]
    Suppose that $K$ is an ellipse. Then there exists an invertible linear transformation $T$ such that $TK=B_2^2$. Observe that 
    \begin{align*}
        d^{(K)}(S)=d(TS)
    \end{align*}
    for any set $S$, and that $T(A+B)=T(A)+T(B)$. Then by \eqref{eq:dyn_farkhi_conjecture}, we have
    \begin{align*}
        d^{(K)}(A+B)^2&=d(T(A)+T(B))^2\\
        &\leq d(T(A))^2+d(T(B))^2\\
        &=d^{(K)}(A)^2+d^{(K)}(B)^2.
    \end{align*}

    To prove the other direction, suppose that $K$ is not an ellipse and let $\mathcal{E}(K)$ be its John ellipse. By Lemma \ref{lemma:boundary_four_points}, the set $E:=\partial K\cap \partial \mathcal{E}(K)$ contains at least three points. Then by Lemma \ref{lemma:arc_lemma}, there exist points $p,x\in E$ such that $\textup{arc}_K(p,x)\cap E=\varnothing$, $p\neq -x$, and $p\in \textup{arc}_K(-x,x)$. In particular, the points $-x,p,x$ are distinct and occur on $\partial K$ in counterclockwise order, so by Lemma \ref{lemma:formula_intersection_with_boundary}, there exists $k\in \partial K\backslash \mathcal{E}(K)$ with $p+x=\|p+x\|_Kk$. Since $k\notin \mathcal{E}(K)$, we have $\|k\|_{\mathcal{E}(K)}>\|k\|_K$. It follows that if we set $u:=p+x$ and $v:=p-x$, then
    \begin{align*}
        \|u\|_{\mathcal{E}(K)}>\|u\|_K,\qquad \|v\|_{\mathcal{E}(K)}\geq \|v\|_K.
    \end{align*}
    The inequality on the right follows directly from the fact that $\mathcal{E}(K)\subset K$. Since $p,x\in \partial \mathcal{E}(K)$, the parallelogram law gives $\|u\|_{\mathcal{E}(K)}^2+\|v\|_{\mathcal{E}(K)}^2=4$, and so
    \begin{align}\label{eq:sharp_bound_ellipse_nonellipse}
        \|u\|_K^2+\|v\|_K^2<4.
    \end{align}
    Set $A:=\{x,p\}$ and $B:=\{0,-x-p\}$. Then
    \begin{align*}
        2d^{(K)}(A)=\|v\|_K,\qquad 2d^{(K)}(B)=\|u\|_K.
    \end{align*}
    We also have $A+B=\{x,p,-x,-p\}\subset\partial K$, and $0\in \textup{conv}(A+B)$, so that $d^{(K)}(A+B)\geq 1$. Then from \eqref{eq:sharp_bound_ellipse_nonellipse}, we have
    \begin{align*}
        d^{(K)}(A)^2+d^{(K)}(B)^2=\frac{\|u\|_K^2+\|v\|_K^2}{4}<1\leq d^{(K)}(A+B)^2,
    \end{align*}
    proving that subadditivity fails when $K$ is not an ellipse.
\end{proof}

\section{Proof of Theorem \ref{theorem:equality_conditions}}

\begin{lemma}\label{lemma:one_dimensional_one_non_convex}
    Let $A,B\subset\mathbb{R}^2$ be compact one-dimensional sets such that one of them is nonconvex. Then
    \begin{align*}
        d(A+B)^2=d(A)^2+d(B)^2
    \end{align*}
    if and only if $\textup{aff}(A)$ and $\textup{aff}(B)$ are orthogonal lines. 
\end{lemma}

\begin{proof}
    By translating $A$ and $B$ independently, we may assume that $0\in A\cap B$, and write $\textup{aff}(A)=\textup{span}(u)$ and $\textup{aff}(B)=\textup{span}(v)$ for unit vectors $u,v$.

    If $u\bot v$, then choose $ru\in \textup{conv}(A)$ and $sv\in \textup{conv}(B)$ that attain the supremum in the definition of $d(A)$ and $d(B)$. Then
    \begin{align*}
        d(ru+sv,A+B)^2=\inf_{r_1u\in A,s_1v\in B}(|r-r_1|^2+|s-s_1|^2)=d(A)^2+d(B)^2.
    \end{align*}
    Equality follows immediately.

    Next, assume that $c:=\langle u,v\rangle \neq 0$. We will show that the distance from any $x\in \textup{conv}(A+B)$ to $A+B$ is strictly smaller than $d(A)^2+d(B)^2$. First, suppose that both sets are nonconvex, and write $x=ru+sv$. If $x\in \textup{conv}(A)+B$ or if $x\in A+\textup{conv}(B)$, then
    \begin{align*}
        d(x,A+B)^2\leq \max\{d(A)^2,d(B)^2\}<d(A)^2+d(B)^2.
    \end{align*}
    Otherwise, choose
    \begin{align*}
        r_1<r<r_2,\qquad r_iu\in A,\qquad r_2-r_1\leq 2d(A),
    \end{align*}
    and similarly, choose $s_1< s< s_2$, $s_iv\in B$, $s_2-s_1\leq 2d(B)$. Let $s_j$ be the closest to $s$, and set
    \begin{align*}
        t:=r-r_1,\qquad \gamma:=r_2-r_1,\qquad q:=s-s_j.
    \end{align*}
    Then $0<t<\gamma \leq 2d(A)$ and $0<|q|\leq d(B)$. We claim that the squared distance from $x$ to at least one of the points $r_1u+s_jv$ and $r_2u+s_jv$ is strictly less than $d(A)^2+q^2$. If not, then
    \begin{align*}
        t^2+2tqc\geq d(A)^2,\qquad (\gamma-t)^2-2(\gamma-t)qc\geq d(A)^2.
    \end{align*}
    Divide the left by $t$, the right by $\gamma-t$, and then add to get
    \begin{align*}
        \gamma\geq d(A)^2\left(\frac{1}{t}+\frac{1}{\gamma-t}\right)\geq\frac{4d(A)^2}{\gamma}.
    \end{align*}
    Then $\gamma\geq 2d(A)$ and as a result the above inequalities are all equality. This forces $\gamma=2d(A)$, $t=d(A)$, and $qc=0$, which is a contradiction. Therefore,
    \begin{align*}
        d(x,A+B)^2<d(A)^2+q^2\leq d(A)^2+d(B)^2.
    \end{align*}

    To complete the proof, we consider the case where $A$ is convex and $B$ is nonconvex, and the other case follows by a symmetry argument. Then $x=ru+sv$, with $ru\in A$. If $sv\in B$, the distance from $x$ to $A+B$ is zero. Otherwise, choose $s_1,s_2$ as above. The closest point $s_jv$ has distance to $B$ strictly less than $d(B)$ unless $s$ is the midpoint between $s_1,s_2$ and $s_2-s_1=2d(B)$. In this case, choose another point $ru+\varepsilon \delta u$ in $A$ in direction $\delta u$, $\delta\in \{-1,1\}$ at small distance $\varepsilon$ from $ru$, and choose $s_jv$ so that $\delta(s-s_j)c>0$. If $\varepsilon <2|(s-s_j)c|$, then
    \begin{align*}
        \|ru+sv-((ru+\varepsilon \delta u )+s_jv)\|_2^2=\varepsilon^2+d(B)^2-2\varepsilon|(s-s_j)c|<d(B)^2.
    \end{align*}
    By the continuity of the norm $\|\cdot \|_2$ and the compactness of $A+B$, the distance $d(A+B)$ is attained for some $x\in \textup{conv}(A+B)$. Therefore, the pointwise strict inequality that we proved above is actually a strict inequality for $d(A+B)$, so equality is impossible.
\end{proof}

\begin{lemma}\label{lemma:convex_nonconvex_two_one}
    Let $A,B\subset\mathbb{R}^2$ be compact sets such that one of them is convex and two-dimensional, and the other nonconvex and one-dimensional. Then
    \begin{align*}
        d(A+B)<\max\{d(A),d(B)\}.
    \end{align*}
\end{lemma}

\begin{proof}
    Suppose that $A$ is convex and two-dimensional and $B$ is nonconvex and one-dimensional. If $x=a+b\in A+\textup{conv}(B)$, there exists a nondegenerate interval $I\subset A$ containing $a$ whose direction is not orthogonal to $\textup{aff}(B)$. Apply Lemma \ref{lemma:one_dimensional_one_non_convex} to the interval $I$ and the nonconvex $B$ to get
    \begin{align*}
        d(x,A+B)\leq d(x,I+B)\leq d(I+B)<d(B).
    \end{align*}
    The result follows by applying the same compactness argument at the end of Lemma \ref{lemma:one_dimensional_one_non_convex} and taking the maximum over all $x\in A+\textup{conv}(B)$.
\end{proof}

\begin{lemma}\label{lemma:both_nonconvex_oneset_two_dimensional}
    Let $A,B\subset\mathbb{R}^2$ be compact nonconvex sets, with at least one being two-dimensional. Then
    \begin{align*}
        d(A+B)^2<d(A)^2+d(B)^2.
    \end{align*}
\end{lemma}

\begin{proof}
    Fix $x\in \textup{conv}(A+B)$. If
    \begin{align*}
        x\in (\textup{conv}(A)+B)\cup (A+\textup{conv}(B)),
    \end{align*}
    then 
    \begin{align*}
        d(x,A+B)^2\leq \max\{d(A)^2,d(B)^2\}<d(A)^2+d(B)^2.
    \end{align*}
    Otherwise, it follows from \cite[Lemma 3.8]{meyer2} (see also \cite{M-2026}) that there exists $\tilde{A}:=\{a_1,a_2\}\subset A$ and $\tilde{B}:=\{b_1,b_2\}\subset B$ such that 
    \begin{align*}
        x\in \textup{conv}(\tilde{A}+\tilde{B}),\qquad d(\tilde{A})\leq d(A),\qquad d(\tilde{B})\leq d(B).
    \end{align*}
    If $\textup{aff}(\tilde{A})$ and $\textup{aff}(\tilde{B})$ are not orthogonal, then by Lemma \ref{lemma:one_dimensional_one_non_convex}, we have
    \begin{align*}
        d(x,A+B)^2\leq d(\tilde{A}+\tilde{B})^2<d(A)^2+d(B)^2.
    \end{align*}
    Suppose that $\textup{aff}(\tilde{A})$ and $\textup{aff}(\tilde{B})$ are orthogonal. Translating if needed, we assume that there exist orthonormal vectors $u$ and $v$ such that 
    \begin{align*}
        \tilde{A}=\{0,r_2u\},\qquad \tilde{B}=\{0,s_2v\},\qquad x=ru+sv,
    \end{align*}
    for $0=:r_1\leq r\leq r_2$ and $0=:s_1\leq s\leq s_2$. Taking the closest distance to a vertex gives
    \begin{align*}
        d(x,A+B)^2\leq \min_{r_i\in \{0,r_2\},s_i\in\{0,s_2\}}(|r-r_i|^2+|s-s_i|^2)\leq d(A)^2+d(B)^2.
    \end{align*}
    The right-side inequality is strict unless
    \begin{align*}
        r=d(A),\qquad r_2=2d(A)\qquad s=d(B),\qquad s_2=2d(B). 
    \end{align*}
    That is, $\tilde{A}+\tilde{B}$ is the vertex set of a rectangle whose sides have maximum possible length, and $x$ is the center of the rectangle.

    Assume without loss of generality that $B$ is two-dimensional. For $\varepsilon>0$ and $\delta\in \{-1,1\}$, define 
    \begin{align*}
        c_{\varepsilon}:=d(B)v+\delta\varepsilon u.
    \end{align*}
    Since $d(B)v$ is the midpoint of the interval $[0,2d(B)v]\subset \textup{conv}(B)$ and $B$ is two-dimensional, we must have $c_\varepsilon\in \textup{conv}(B)$ for some $\delta$ and sufficiently small $\varepsilon$. Choose $p_{\varepsilon}\in B$ with 
    \begin{align*}
        \|p_\varepsilon-c_\varepsilon\|_2\leq d(B).
    \end{align*}
    Write 
    \begin{align*}
        p_\varepsilon=(\delta\varepsilon+a_\varepsilon)u+(d(B)+b_\varepsilon)v,\qquad a_\varepsilon^2+b_\varepsilon^2\leq d(B)^2.
    \end{align*}
    Taking the distance from $x$ to either $p_\varepsilon$ or $2d(A)u+p_\varepsilon$ gives
    \begin{align}\label{eq:epsilon_upper_bound}
        d(x,A+B)^2\leq (d(A)-|\delta\varepsilon +a_\varepsilon|)^2+b_\varepsilon^2.
    \end{align}
    Consider a sequence $\varepsilon\rightarrow 0$ along which $(a_\varepsilon,b_\varepsilon)$ converges to a point $(a,b)$. If $a\neq 0$, then the limit of the right side of \eqref{eq:epsilon_upper_bound} is
    \begin{align*}
        d(A)^2-2d(A)|a|+a^2+b^2<d(A)^2+d(B)^2.
    \end{align*}
    If $a= 0$, then for small $\varepsilon$ the quantity $|\delta\varepsilon +a_\varepsilon|$ is strictly less than $d(A)$. Then for $\varepsilon$ sufficiently small, either the quantity $|\delta\varepsilon +a_\varepsilon|$ is nonzero and the first term on the right of \eqref{eq:epsilon_upper_bound} is strictly less than $d(A)^2$, or it is zero so that $a_\varepsilon\neq 0$, and as a result $b_{\varepsilon}^2<d(B)^2$. It follows that the strict bound
    \begin{align*}
        d(x,A+B)^2<d(A)^2+d(B)^2
    \end{align*}
    holds for every $x\in \textup{conv}(A+B)$. The compactness of $A$ and $B$ implies that there exists an $x$ that achieves the maximum, which proves the desired strict bound.
\end{proof}

\begin{proof}[Proof of Theorem \ref{theorem:equality_conditions}]

First, we will prove the theorem for the case where $E=B_2^2$. If the sets $\textup{aff}(A)$ and $\textup{aff}(B)$ are orthogonal lines, then it follows from Lemma \ref{lemma:one_dimensional_one_non_convex} that equality holds. 

For the other direction, suppose that equality holds. By Lemma \ref{lemma:both_nonconvex_oneset_two_dimensional}, both sets are one-dimensional. It follows from Lemma \ref{lemma:one_dimensional_one_non_convex} that $\textup{aff}(A)$ and $\textup{aff}(B)$ are orthogonal lines. 

Now, we consider the case of an arbitrary ellipse $E$ centered at $0$. There exists an invertible linear transformation $T$ such that $TE=B_2^2$. Similar to the observation in the proof of Theorem \ref{theorem:norm_rigidity_theorem}, if $S\subset\mathbb{R}^2$ is a nonempty set, then
\begin{align*}
    d^{(E)}(S)=d(TS).
\end{align*}
It follows that equality holds for the distance $d^{(E)}$ with the sets $A,B$ if and only if equality holds for the distance $d$ with the sets $TA,TB$. That is, if and only if $\textup{aff}(TA)$ and $\textup{aff}(TB)$ are orthogonal lines, which is equivalent to saying that $\textup{aff}(A)$ and $\textup{aff}(B)$ are $E$-orthogonal lines.
\end{proof}

\begin{proof}[Proof of Proposition \ref{proposition:convex+nonconvex}]
    Using the formula $d^{(E)}(S)=d(TS)$ that we observed in the proof of Theorem \ref{theorem:equality_conditions}, it is enough to prove this result for the case $E=B_2^2$. If $B$ is two-dimensional, then we have nothing to prove. If $B$ is one-dimensional, then by Lemma \ref{lemma:convex_nonconvex_two_one}, $A$ must also be one-dimensional. But then by Lemma \ref{lemma:one_dimensional_one_non_convex}, the sets $\textup{aff}(A)$ and $\textup{aff}(B)$ are orthogonal lines. 
\end{proof}

\newpage
\bibliography{hausdorff_distance}
\bibliographystyle{abbrv}

\bigskip
\noindent Mark Meyer
\\
LAMA, Univ Gustave Eiffel, Univ Paris Est Creteil, 77447 Marne-la-Vall\'ee, France.
\\
E-mail address: mark.meyer@univ-eiffel.fr
\vspace{2mm}
\\

\end{document}